\documentclass{amsart}

\usepackage{amsthm,amssymb,amsfonts,latexsym,mathtools,thmtools}
\usepackage[T1]{fontenc}
\usepackage{tikz-cd} 
\usepackage{enumitem} 
\usepackage{hyperref} 
\usepackage{hyperref}
\hypersetup{
    colorlinks=true,
    linkcolor=blue,
    filecolor=blue,      
    urlcolor=cyan,
    linktocpage=true
}
\newtheorem{theorem}{Theorem}[section]
\newtheorem{lemma}[theorem]{Lemma}

\theoremstyle{definition}
\newtheorem{definition}[theorem]{Definition}
\newtheorem{example}[theorem]{Example}

\newtheorem{proposition}[theorem]{Proposition}

\newtheorem{corollary}[theorem]{Corollary}

\theoremstyle{remark}

\numberwithin{equation}{section}

\begin{document}

\title[Types of elements, Gelfand and strongly harmonic rings]{On types of elements, Gelfand and strongly harmonic rings of skew PBW extensions over weak compatible rings}

\author{Andr\'es Chac\'on}
\address{Universidad Nacional de Colombia - Sede Bogot\'a}
\curraddr{Campus Universitario}
\email{anchaconca@unal.edu.co}
\thanks{}

\author{Sebasti\'an Higuera}
\address{Universidad Nacional de Colombia - Sede Bogot\'a}
\curraddr{Campus Universitario}
\email{sdhiguerar@unal.edu.co}
\thanks{}

\author{Armando Reyes}
\address{Universidad Nacional de Colombia - Sede Bogot\'a}
\curraddr{Campus Universitario}
\email{mareyesv@unal.edu.co}
\thanks{The authors were supported by the research fund of Faculty of Science, Code HERMES 53880, Universidad Nacional de Colombia--Sede Bogot\'a, Colombia.}

\date{}
\dedicatory{Dedicated to Professor Oswaldo Lezama}

\begin{abstract}

We investigate and characterize several kinds of elements such as units, idempotents, von Neumann regular, $\pi$-regular and clean elements for skew PBW extensions over weak compatible rings. We also study the notions of Gelfand and Harmonic rings for these families of algebras. The results presented here extend those corresponding in the literature for commutative and noncommutative rings of polynomial type.

\end{abstract}
\maketitle

\section{Introduction}\label{ch0}


Throughout the paper, the term ring  means an associative (not necessarily commutative) ring with identity unless otherwise stated. Additionally, for a ring $R$, the sets $N_{*}(R)$, $N(R)$, $N^{*}(R)$, $U(R)$, $Z(R)$ and $J(R)$ denote the prime radical, the set of nilpotent elements, the upper nil radical (i.e., the sum of all its nil ideals of $R$), the set of units, the center, and the Jacobson radical of $R$, respectively. 

\medskip

Recall that a ring $R$ is {\em reduced} if it has no nonzero nilpotent elements, that is, $N(R)= \{ 0 \}$. On the other hand, a ring $R$ is called {\em semicommutative} if $ab = 0$ implies $aRb = 0$, for any elements $a, b\in R$. If the equality $N(R)^{*} = N(R)$ holds, Marks \cite{Marks2001} called $R$ an {\em NI} ring. If $J(R)=N(R)$, then we say that $R$ is an {\em NJ ring}. The following relations are well-known: reduced $\Rightarrow$ semicommutative $\Rightarrow$  NI, but the converses do not hold (see \cite{ChenCui2011} and \cite{Karamzadeh1987} for more details). A ring $R$ is called {\em Abelian} if every idempotent is central. A ring $R$ is called {\em right} (respectively, {\em left}) {\em duo} if every right (respectively, left) ideal is an ideal. It follows that the relations one sided duo (left or right) $\Rightarrow$ semicommutative $\Rightarrow$ Abelian hold.

\medskip

In 1936, von Neumann \cite{vonNeumann1936} introduced the von Neumann regular rings as an algebraic tool for studying certain lattices and some properties of operator algebras. Briefly, a ring $R$ is called {\em von Neumann regular} if for every element $a\in R$ there exists an element $r\in R$ such that $a = ara$.  These rings are also known as {\em absolutely flat rings} due to their characterization in terms of modules. von Neumann regular rings are of great importance in areas such as topology and functional analysis. More precisely, the prime spectrum of a commutative von Neumann ring establishes relationships with different types of compactifications and homomorphisms of the prime spectrum and the prime spectrum of its ring of idempotent elements (see \cite{Goodearl1979, RubioAcosta2012, Rump2010} for more details). These facts show the close relationship of the von Neumann rings with Boolean rings.

\medskip

Related to the description of idempotent elements and other kinds of elements of a ring, we can find the Gelfand rings and the clean elements. Recall that a ring $R$ is said to be a {\em Gelfand ring} if for each pair of distinct maximal right ideals $M_1$, $M_2$ of $R$, there exist two right ideals $I_1$, $I_2$ of $R$ such that $I_1\not\subseteq M_1$, $I_2\not\subseteq M_2$ and $I_1I_2=0$. On the other hand, an element $a\in R$ is called a {\em clean element} if $a$ is the sum of a unit and an idempotent of $R$. $R$ is said to be a {\em clean ring} if every element of $R$ is clean. In commutative algebra, it is known that every clean ring is Gelfand. Additionally, Gelfand rings are tied to the Zariski topology over a ring, which allows to characterize different properties of the prime spectrum and the maximal spectrum of a ring (for more details, see \cite{Aghajanietal2020}).

\medskip

Continuing with the study of Gelfand rings and their relationship with topological spaces, Mulvey \cite{Mulvey1976} obtained a generalization of Swan's theorem concerning vector bundles over a compact topological space. He established an equivalence between the category of modules over a Gelfand ring and the category of modules over the corresponding compact ringed space. The algebraic $K$-theory of commutative Gelfand rings has been studied by Carral \cite{Carral1980} by showing relationships between the stable rank over Gelfand rings and the covering dimension of the maximal ideal space. These results are analogous to those corresponding for Noetherian rings with respect to the Krull dimension. Zhang et al. \cite{ZhangGelfand} showed that certain topological spaces related to $R$ are normal if and only if the quotient ring $R/N_*(R)$ is Gelfand.


\medskip


Concerning noncommutative rings of polynomial type, for the {\em skew polynomial rings} (also known as {\em Ore extensions})  $R[x;\sigma,\delta]$ introduced by Ore \cite{Ore1933}, Hashemi et al., \cite{Hashemi} investigated characterizations of different elements over skew polynomial rings by using the notion of {\em compatible ring} defined by Annin \cite{Annin2004} (c.f. Hashemi and Moussavi \cite{HashemiMoussavi2005}). Briefly, if $R$ is a ring, $\sigma$ is an endomorphism of $R$, and $\delta$ is a $\sigma$-derivation of $R$, then (i) $R$ is said to be $\sigma$-{\em compatible} if for each $a, b\in R$, $ab = 0$ if and only if $a\sigma(b)=0$ (necessarily, the endomorphism $\sigma$ is injective). (ii) $R$ is called $\delta$-{\em compatible} if for each $a, b\in R$, $ab = 0$ implies $a\delta(b)=0$. (iii) If $R$ is both $\sigma$-compatible and $\delta$-compatible, then $R$ is called ($\sigma,\delta$)-{\em compatible}. With these ring-theoretical notions, Hashemi et al., \cite{Hashemi} characterized the unit elements, idempotent elements, von Neumann regular elements, $\pi$-regular elements, and also the von Neumann local elements of the skew polynomial ring $R[x; \sigma, \delta]$ when the base ring $R$ is a right duo $(\sigma,\delta)$-compatible.

\medskip

Motivated by the notion of compatibility for Ore extensions, Hashemi et al. \cite{Hashemietal2017} and Reyes and Su\'arez \cite {ReyesSuarez2018} introduced independently the $(\Sigma, \Delta)$-{\em compatible rings} (see Section \ref{weakcompatiblerings}) as a natural generalization of $(\sigma, \delta)$-compatible rings with the aim of studying the {\em skew PBW extensions} defined by Gallego and Lezama \cite{GallegoLezama2011} (in Section \ref{Definitions} we will say some words about the generality of these objects with respect to other noncommutative algebras). Examples, ring and module theoretic properties of these extensions over compatible rings have been investigated by some people (e.g. \cite{Hashemietal2017, HigueraReyes2022, ReyesSuarez2018, ReyesSuarez2020}). In particular, Hamidizadeh et al. \cite{Hamidizadehetal2020} characterized the above types of elements of skew PBW extensions over compatible rings generalizing the results obtained by Hashemi et al. \cite{Hashemi}.

\medskip

Since Reyes and Su\'arez \cite{ReyesSuarez2020} introduced the {\em weak $(\Sigma,\Delta)$-compatible rings} as a natural generalization of compatible rings over skew PBW extensions and the {\em weak $(\sigma,\delta)$-compatible rings} defined by Ouyang and Liu \cite{LunquenJingwang2011} for Ore extensions, and Higuera and Reyes \cite{HigueraReyes2022} characterized the weak annihilators and nilpotent associated primes of skew PBW extensions by using this weak notion, an immediate and natural task is to study the types of elements described above of these extensions by considering this weak notion of compatibility, and hence to investigate if it is possible to extend all results established in \cite{Hamidizadehetal2020, Hashemi} to a more general setting. This is the purpose of the paper.

\medskip

With this aim, the article is organized as follows. In Section \ref{Definitions}, we recall some definitions and results about skew PBW extensions and weak $(\Sigma, \Delta)$-compatible rings. Section \ref{originalresults} contains the original results of the paper. More exactly, Section \ref{Elements} presents results concerning the characterization of different types of elements such as idempotents, units, von Neumman regular, and clean elements of skew PBW extensions over weak compatible rings (Theorems \ref{th.units}, \ref{theoremidem}, \ref{Proposition1.3.3.a}, \ref{WeakTheorem4.14}, \ref{WeakTheorem4.15}, \ref{WeakTheorem4.16} and \ref{WeakTheorem4.17}). Next, in Section \ref{Gelfandrings} we investigate the notions of strongly harmonic and Gelfand rings over skew PBW extensions (Propositions \ref{prop.strongly.coef} and \ref{prop.not.local}, and Theorems \ref{th.no.gelfand} and \ref{theorem.harmonic.unique}). Our results generalize those corresponding for skew PBW extensions over right duo rings presented by Hamidizadeh et al. \cite{Hamidizadehetal2020}, and Ore extensions over noncommutative rings presented by Hashemi et al. \cite{Hashemi}. Finally, Section \ref{future} presents some ideas for possible future research.

\medskip

Throughout the paper, $\mathbb{N}$ and $\mathbb{Z}$ denote the natural and integer numbers. We assume the set of natural numbers including zero. 

\section{Definitions and elementary properties}\label{Definitions}

\subsection{Skew Poincar\'e-Birkhoff-Witt extensions}\label{definitionexamplesspbw}
Skew PBW extensions were defined by Gallego and Lezama \cite{GallegoLezama2011} with the aim of generalizing Poincar\'e-Birkhoff-Witt extensions introduced by Bell and Goodearl \cite{BellGoodearl1988} and Ore extensions of injective type defined by Ore \cite{Ore1933}. Over the years, several authors have shown that skew PBW extensions also generalize families of noncommutative algebras such as 3-dimensional skew polynomial algebras introduced by Bell and Smith \cite{BellSmith1990}, diffusion algebras defined by Isaev et al. \cite{IsaevPyatovRittenberg2001}, ambiskew polynomial rings introduced by Jordan in several papers \cite{Jordan1993, Jordan2000,JordanWells1996}, solvable polynomial rings introduced by Kandri-Rody and Weispfenning  \cite{KandryWeispfenninig1990}, almost normalizing extensions defined by McConnell and Robson \cite{McConnellRobson2001}, skew bi-quadratic algebras recently introduced by Bavula \cite{Bavula2021}, and others. For more details about the relationships between skew PBW extensions and other algebras having PBW bases, see \cite{BrownGoodearl2002, Lezamabook2020, GoodearlWarfield2004, McConnellRobson2001} and references therein. 

\begin{definition}[{\cite[Definition 1]{GallegoLezama2011}}] \label{def.skewpbwextensions}
Let $R$ and $A$ be rings. We say that $A$ is a \textit{skew PBW extension over} $R$ (the ring of coefficients), denoted $A=\sigma(R)\langle
x_1,\dots,x_n\rangle$, if the following conditions hold:
\begin{enumerate}
\item[\rm (i)]$R$ is a subring of $A$ sharing the same identity element.
\item[\rm (ii)] There exist finitely many elements $x_1,\dots ,x_n\in A$ such that $A$ is a left free $R$-module, with basis the
set of standard monomials
\begin{center}
${\rm Mon}(A):= \{x^{\alpha}:=x_1^{\alpha_1}\cdots
x_n^{\alpha_n}\mid \alpha=(\alpha_1,\dots ,\alpha_n)\in
\mathbb{N}^n\}$.
\end{center}
Moreover, $x^0_1\cdots x^0_n := 1 \in {\rm Mon}(A)$.
\item[\rm (iii)]For every $1\leq i\leq n$ and any $r\in R\ \backslash\ \{0\}$, there exists $c_{i,r}\in R\ \backslash\ \{0\}$ such that $x_ir-c_{i,r}x_i\in R$.
\item[\rm (iv)]For $1\leq i,j\leq n$, there exists $d_{i,j}\in R\ \backslash\ \{0\}$ such that
\[
x_jx_i-d_{i,j}x_ix_j\in R+Rx_1+\cdots +Rx_n,
\]
i.e. there exist elements $r_0^{(i,j)}, r_1^{(i,j)}, \dotsc, r_n^{(i,j)} \in R$ with
\begin{center}
$x_jx_i - d_{i,j}x_ix_j = r_0^{(i,j)} + \sum_{k=1}^{n} r_k^{(i,j)}x_k$.    
\end{center}
\end{enumerate}
\end{definition}

Since ${\rm Mon}(A)$ is a left $R$-basis of $A$, the elements $c_{i,r}$ and $d_{i, j}$ are unique. Thus, every nonzero element $f \in A$ can be uniquely expressed as $f = \sum_{i=0}^ma_iX_i$, with $a_i \in R$, $X_0=1$, and $X_i \in \text{Mon}(A)$, for $0 \leq i \leq m$  (when necessary, we use the notation $f = \sum_{i=0}^ma_iY_i$)  \cite[Remark 2]{GallegoLezama2011}. 

\begin{proposition}[{\cite[Proposition 3]{GallegoLezama2011}}] \label{sigmadefinition}
If $A=\sigma(R)\langle x_1,\dots,x_n\rangle$ is a skew PBW extension, then there exist an injective endomorphism $\sigma_i:R\rightarrow R$ and a $\sigma_i$-derivation $\delta_i:R\rightarrow R$ such that $x_ir=\sigma_i(r)x_i+\delta_i(r)$, for each $1\leq i\leq n$, where $r\in R$.
\end{proposition}

We use the notation $\Sigma:=\{\sigma_1,\dots,\sigma_n\}$ and $\Delta:=\{\delta_1,\dots,\delta_n\}$ for the families of injective endomorphisms and $\sigma_i$-derivations, respectively, established in Proposition \ref{sigmadefinition}. For a skew PBW extension $A = \sigma(R)\langle x_1,\dotsc, x_n\rangle$ over $R$, we say that the pair $(\Sigma, \Delta)$ is a \textit{system of endomorphisms and $\Sigma$-derivations} of $R$ with respect to $A$. For $\alpha = (\alpha_1, \dots , \alpha_n) \in \mathbb{N}^n$, $\sigma^{\alpha}:= \sigma_1^{\alpha_1}\circ \cdots \circ \sigma_n^{\alpha_n}$, $\delta^{\alpha} := \delta_1^{\alpha_1} \circ \cdots \circ \delta_n^{\alpha_n}$, where $\circ$ denotes the classical composition of functions.

\medskip

We recall some results about quotient rings of skew PBW extensions which are useful for the paper (c.f. \cite{LezamaAcostaReyes2015}).


\begin{definition}[{\cite[Definition 5.1.1]{Lezamabook2020}}]
    Let $R$ be a ring and $(\Sigma,\Delta)$ a system of endomorphisms and $\Sigma$-derivations of $R$. If $I$ is a two-sided ideal of $R$, then $I$ is called $\Sigma$-{\em invariant} if $\sigma^{\alpha}(I)\subseteq I$, where $\alpha \in \mathbb{N}^n$. $I$ is a $\Delta$-{\em invariant} ideal if $\delta^{\alpha}(I)\subseteq I$, where $\alpha \in \mathbb{N}^n$. If $I$ is both $\Sigma$ and $\Delta$-invariant, we say that $I$ is {\em $(\Sigma,\Delta)$-invariant}.
\end{definition}

\begin{proposition}[{\cite[Proposition 5.1.2]{Lezamabook2020}}]
    Let $R$ be a ring, $(\Sigma,\Delta)$ a system of endomorphisms and $\Sigma$-derivations of $R$, $I$ a proper two-sided ideal of $R$ and $\overline{R}:=R/I$. If $I$ is $(\Sigma,\Delta)$-invariant then over $\overline{R}$ is induced a system $(\overline{\Sigma},\overline{\Delta})$ of endomorphisms and $\overline{\Sigma}$-derivations defined by $\overline{\sigma_i}(\overline{r}):=\overline{\sigma_i(r)}$ and $\overline{\delta_i}(\overline{r}):=\overline{\delta_i(r)}$, $1\le i\le n$.
\end{proposition}

\begin{proposition}[{\cite[Proposition 5.1.6]{Lezamabook2020}}]\label{prop.invariant}
Let $A=\sigma(R)\langle x_1,\dots,x_n\rangle$ be a skew PBW extension and $I$ a $(\Sigma,\Delta)$-invariant ideal of $R$. Then:
\begin{enumerate}
    \item[{\rm (1)}] $IA$ is an ideal of $A$ and $IA\cap R=I$. $IA$ is proper if and only if $I$ is proper. Moreover, if for every $1\leq i\leq n$, $\sigma_i$ is bijective and $\sigma_i(I)=I$, then $IA=AI$.
    \item[{\rm (2)}] If $I$ is proper and $\sigma_i(I)=I$ for every $1\leq i\leq n$, then $A/IA$ is a skew PBW extension of $R/I$.
\end{enumerate}
\end{proposition}

The following result shows the relation between NI rings and invariant ideals.

\begin{proposition}\label{prop.N.invariante}
Let $A=\sigma(R)\langle x_1,\dots,x_n\rangle$ be a skew PBW extension. If $A$ is NI then $N(R)$ is a $(\Sigma,\Delta)$-invariant ideal.
\begin{proof}
Clearly, we have that $N(R)$ is a $\Sigma$-invariant ideal. On the other hand, consider an element $a\in N(R)$. Since $a$ and $\sigma_i(a)$ are elements of $N(A)$ and $N(A)$ is an ideal of $A$, then $\delta_i(a)=x_ia-\sigma_i(a)x_i\in N(A)$, that is, $\delta_i(a)\in N(R)$. Therefore, $N(R)$ is a $(\Sigma,\Delta)$-invariant ideal.
\end{proof}
\end{proposition}

\subsection{Weak $(\Sigma,\Delta)$-compatible rings}\label{weakcompatiblerings}

Let $R$ be a ring and $\sigma$ an endomorphism of $R$. Krempa \cite{Krempa1996} defined $\sigma$ as a {\em rigid endomorphism} if $r\sigma(r)=0$ implies $r=0$, where $r\in R$. In this way, a ring $R$ is called $\sigma$-{\em rigid} if there exists a rigid endomorphism $\sigma$ of $R$. Following Annin \cite{Annin2004} (c.f. Hashemi and Moussavi \cite{HashemiMoussavi2005}), a ring $R$ is said to be $\sigma$-{\em compatible} if for every $a, b \in R$, $ab = 0$ if and only if $a\sigma(b) = 0$; $R$ is called $\delta$-{\em compatible}, if for each $a, b \in R$, we have $ab = 0 \Rightarrow a\delta(b) = 0$. Moreover, if $R$ is both $\sigma$-compatible and $\delta$-compatible, then $R$ is said to be $(\sigma,\delta)$-{\em compatible}. Reyes and Su\'arez \cite {ReyesSuarez2018} and Hashemi, Khalilnezhad and Alhevaz \cite{Hashemietal2017} introduced independently the $(\Sigma, \Delta)$-compatible rings which are a natural generalization of $(\sigma, \delta)$-compatible rings. Briefly, for a ring $R$ with a finite family of endomorphisms $\Sigma$ and a finite family of $\Sigma$-derivations $\Delta$, and considering the notation established in Proposition \ref{sigmadefinition} for families of endomorphisms and derivations, we say that $R$ is $\Sigma$-{\em compatible} if for each $a, b \in R$, $a\sigma^{\alpha}(b) = 0$ if and only if $ab = 0$, where $\alpha \in \mathbb{N}^n$. Similarly, we say that $R$ is $\Delta$-{\em compatible} if for each $a, b \in R$, it follows that $ab = 0$ implies $a\delta^{\beta}(b)=0$, where $\beta \in \mathbb{N}^n$. If $R$ is both $\Sigma$-compatible and $\Delta$-compatible, then $R$ is called $(\Sigma, \Delta)$-{\em compatible}. 

\medskip

Examples of skew PBW extensions over $(\Sigma, \Delta)$-compatible rings include PBW extensions defined by Bell and Goodearl \cite{BellGoodearl1988}, some operator algebras (e.g., the algebra of linear partial differential operators, the algebra of linear partial shift operators, the algebra of linear partial difference operators, the algebra of linear partial $q$-dilation operators, and the algebra of linear partial $q$-differential operators), the class of diffusion algebras \cite{IsaevPyatovRittenberg2001}, quantizations of Weyl algebras, the family of 3-dimensional skew polynomial algebras \cite{BellSmith1990}, and other families of noncommutative algebras having PBW bases. A detailed list of examples can be found in \cite{Hashemietal2017, HigueraReyes2022, Jordan2000, ReyesSuarez2018}.

\medskip

The next definition present the {\em weak compatible rings} which are more general than compatible rings (see Examples \ref{exampleweak1} and \ref{exampleweak2}).

\begin{definition}[{\cite[Definition 4.1]{ReyesSuarez2020}}]\label{def.weakcom}
Let $R$ be a ring with a finite family of endomorphisms $\Sigma$ and a finite family of $\Sigma$-derivations $\Delta$. We say that $R$ is {\it weak $\Sigma$-compatible} if for each $a,b\in R$, $a\sigma^\alpha(b)\in N(R)$ if and only if $ab\in N(R)$, for all $\alpha\in\mathbb{N}^n$. Similarly, $R$ is called {\it weak $\Delta$-compatible} if for each $a,b\in R$, $ab\in N(R)$ implies $a\delta^\beta(b)\in N(R)$, for every $\beta\in\mathbb{N}^n$. If $R$ is both weak $\Sigma$-compatible and weak $\Delta$-compatible, then $R$ is called {\it weak $(\Sigma,\Delta)$-compatible}.
\end{definition}

\begin{proposition}[{\cite[Proposition 4.2]{ReyesSuarez2020}}]\label{prop.weak}
If $R$ is a weak $(\Sigma,\Delta)$-compatible ring, then the following assertions hold:
\begin{enumerate}
\item If $ab\in N(R)$, then $a\sigma^\alpha(b),\sigma^\beta(a)b\in N(R)$, for all elements $\alpha,\beta\in\mathbb{N}^n$.
\item If $\sigma^\alpha(a)b\in N(R)$, for some element $\alpha\in\mathbb{N}^n$, then $ab\in N(R)$.
\item If $a\sigma^\beta(b)\in N(R)$, for some element $\beta\in\mathbb{N}^n$, then $ab\in N(R)$.
\item If $ab\in N(R)$, then $\sigma^\alpha(a)\delta^\beta(b),\delta^\beta(a)\sigma^\alpha(b)\in N(R)$, for every $\alpha,\beta\in\mathbb{N}^n$.
\end{enumerate}
\end{proposition}

We present two examples of skew PBW extensions over weak $(\Sigma,\Delta)$-compatible rings that are not $(\Sigma,\Delta)$-compatible.

\begin{example}\label{exampleweak1} Let $R$ be a reduced ring and  $R_2$ the ring of upper triangular matrices. Consider the endomorphism $\sigma:R_2 \rightarrow R_2$ defined by $\sigma \left( \bigl( \begin{smallmatrix}a & b\\ 0 & c \end{smallmatrix} \bigr) \right)= \bigl(\begin{smallmatrix}a & 0\\0 & c\end{smallmatrix}\bigr)$, and the $\sigma$-derivation $\delta: R_2 \rightarrow R_2$ defined by $\delta \left( \bigl( \begin{smallmatrix}a & b\\ 0 & c \end{smallmatrix} \bigr) \right)= \bigl( \begin{smallmatrix}0 & b\\ 0 & 0 \end{smallmatrix} \bigr)$, for all $\bigl( \begin{smallmatrix}a & b\\ 0 & c \end{smallmatrix} \bigr) \in R_2$. Notice that $\bigl(\begin{smallmatrix} 1 & 1 \\ 0 & 1\end{smallmatrix}\bigr)\cdot \sigma \left(\bigl(\begin{smallmatrix}0 & 1\\ 0 &0 \end{smallmatrix}\bigr) \right)=\bigl(\begin{smallmatrix}0 & 0\\ 0 & 0\end{smallmatrix}\bigr)$ with $\bigl(\begin{smallmatrix} 1 & 1 \\ 0 & 1\end{smallmatrix}\bigr)\cdot \bigl(\begin{smallmatrix}0 & 1\\ 0 &0 \end{smallmatrix}\bigr) \neq \bigl(\begin{smallmatrix}0 & 0\\ 0 & 0\end{smallmatrix}\bigr)$ which means that $R_2$ is not a $(\sigma, \delta)$-compatible ring. On the other hand, the set of nilpotent elements of $R_2$ consists of all matrices of the form $\bigl(\begin{smallmatrix}0 & b\\ 0 & 0\end{smallmatrix}\bigr)$, for any element $b \in R$. In this way, if $\bigl(\begin{smallmatrix} a & b \\ 0 & c\end{smallmatrix}\bigr)\cdot \bigl(\begin{smallmatrix}e & f\\ 0 & h \end{smallmatrix}\bigr) \in N(R_2)$, then $ae=ch=0$. This implies that $\bigl(\begin{smallmatrix} a & b \\ 0 & c\end{smallmatrix}\bigr)\cdot \sigma \left( \bigl( \begin{smallmatrix} e & f\\ 0 & h \end{smallmatrix} \bigr) \right) = \bigl(\begin{smallmatrix} a & b \\ 0 & c\end{smallmatrix}\bigr)\cdot  \bigl(\begin{smallmatrix}e & 0\\ 0 & h \end{smallmatrix}\bigr) \in N(R_2)$. By a similar argument, if $A\sigma(B) \in N(R_2)$, then $AB \in N(R_2)$, for all $A, B \in R_2$. Finally, $A\delta(B) \in N(R_2)$, for all $A,B \in R_2$ with $AB \in N(R_2)$. Therefore, we conclude $R_2$ is a weak $(\sigma, \delta)$-compatible ring. Notice that the Ore extension $R_2[x;\sigma,\delta]$ is a skew PBW extension over $R_2$ which is weak $(\sigma,\delta)$-compatible.
\end{example}

\begin{example}[{\cite[Example 3.2]{Suarezetal2021RNP}}]\label{exampleweak2} Let ${\rm S}_2(\mathbb{Z})$ be a subring of upper triangular matrices defined by ${\rm S}_2(\mathbb{Z})= \left \{ \bigl(\begin{smallmatrix}a & b\\ 0 & a \end{smallmatrix}\bigr) \mid a,b \in \mathbb{Z} \right \}$. Let $\sigma_1={\rm id}_{{\rm S}_2(\mathbb{Z})}$ be the identity endomorphism of ${\rm S}_2(\mathbb{Z})$, and consider $\sigma_2$ and $\sigma_3$ two endomorphisms defined by 
    $\sigma_2 \left( \bigl( \begin{smallmatrix}a & b\\ 0 & a \end{smallmatrix} \bigr) \right)= \bigl(\begin{smallmatrix}a & -b\\0 & a\end{smallmatrix}\bigr)$ and $\sigma_3\left( \bigl( \begin{smallmatrix}a & b\\ 0 & a \end{smallmatrix} \bigr) \right)= \bigl( \begin{smallmatrix}a & 0\\0 & a\end{smallmatrix} \bigr)$.
Note that ${\rm S}_2(\mathbb{Z})$ is not $\sigma_3$-compatible, since for $A=\bigl( \begin{smallmatrix} 1 & 1\\ 0& 1\end{smallmatrix} \bigr)$ and $B=\bigl(\begin{smallmatrix} 0 & 1\\ 0 & 0 \end{smallmatrix}\bigr)$, we have $A\sigma_3(B)=0$ but $AB=\bigl( \begin{smallmatrix} 0 & 1\\ 0& 0\end{smallmatrix}\bigr) \neq 0$. Hence, we conclude ${\rm S}_2(\mathbb{Z})$ is not a $\Sigma$-compatible ring. In the same way, the set of nilpotent elements of $S_2(\mathbb{Z})$ consists of all matrices of the form $\bigl(\begin{smallmatrix}0 & b\\ 0 & 0\end{smallmatrix}\bigr)$, for any $b \in \mathbb{Z}$. An argument similar to the previous example shows that ${\rm S}_2(\mathbb{Z})$ is a weak $\Sigma$-compatible ring, so we can consider a skew PBW extension $A= \sigma({\rm S}_2(\mathbb{Z}))\langle x,y,z \rangle$ with three indeterminates $x, y$ and $z$ satisfying the conditions established in Definition \ref{def.skewpbwextensions}.
\end{example}

Proposition \ref{prop.rigid.and.weak} shows that if $R$ is reduced, then the notions of compatible ring and weak compatible ring coincide (c.f. \cite[Theorem 3.9]{ReyesSuarez2018}).

\begin{proposition}[{\cite[Theorem 4.5]{ReyesSuarez2020}}]\label{prop.rigid.and.weak}
If $A=\sigma(R)\langle x_1,\dots,x_n\rangle$ is a skew PBW extension, then the following statements are equivalent:\begin{enumerate}
    \item $R$ is reduced and weak $(\Sigma,\Delta)$-compatible.
    \item $R$ is $\Sigma$-rigid.
    \item $A$ is reduced.
\end{enumerate}  
\end{proposition}

Ouyang and Liu \cite{LunquenJingwang2011} characterized the nilpotent elements of skew polynomial rings over a weak $(\sigma, \delta)$-compatible and NI ring. Reyes and Su\'arez \cite{ReyesSuarez2020} extended this result for skew PBW extensions as the following proposition shows. We assume that the elements $d_{i,j}$ from Definition \ref{def.skewpbwextensions}(iv) are central and invertible in $R$.

\begin{proposition}[{\cite[Theorem 4.6]{ReyesSuarez2020}}]\label{prop.nilp}
If $A=\sigma(R)\langle x_1,\dots,x_n\rangle$ is a skew PBW extension over a weak $(\Sigma,\Delta)$-compatible and NI ring, then $f=\sum_{i=0}^m a_iX_i\in N(A)$ if and only if $a_i\in N(R)$, for all $0\leq i\leq m$. 
\end{proposition}

Ouyang et al. \cite{Lunqunetal2013} introduced the notion of skew $\pi$-Armendariz ring as follows: If $R$ is a ring with an endomorphism $\sigma$ and a $\sigma$-derivation $\delta$, then $R$ is called {\em skew $\pi$-Armendariz ring} if for polynomials $f(x) = \sum_{i=0}^{l} a_ix^i$ and $g(x) = \sum_{j=0}^{m} b_jx^j$ of $R[x; \sigma, \delta]$, $f(x)g(x) \in N(R[x; \sigma, \delta])$ implies that $a_ib_j \in N(R)$, for each $0 \leq i \leq l$ and $0 \leq j \leq m$. Skew $\pi$-Armendariz rings are more general than skew Armendariz rings when the ring of coefficients is $(\sigma,\delta)$-compatible \cite[Theorem 2.6]{Lunqunetal2013}, and also extend $\sigma$-Armendariz rings defined by Hong et al. \cite{Hongetal2006} considering $\delta$ as the zero derivation.

\medskip

Ouyang and Liu \cite{LunquenJingwang2011} showed that if $R$ is a weak $(\sigma,\delta)$-compatible and NI ring, then $R$ is skew $\pi$-Armendariz ring  \cite[Corollary 2.15]{LunquenJingwang2011}. Reyes \cite{Reyes2018} formulated the analogue of skew $\pi$-Armendariz ring in the setting of skew PBW extensions. For a skew PBW extension $A = \sigma(R)\langle x_1,\dotsc, x_n\rangle$ over a ring $R$, we say that $R$ is {\em skew} $\Pi$-{\em Armendariz ring} if for elements $f = \sum_{i=0}^l a_iX_i$ and $g = \sum_{j=0}^m b_jY_j$ belong to $A$, $fg \in N(A)$ implies $a_ib_j \in N(R)$, for each $0 \leq i \leq l$ and $0 \leq j \leq m$. If $R$ is reversible and $(\Sigma,\Delta)$-compatible, then $R$ is skew $\Pi$-Armendariz ring \cite[Theorem 3.10]{Reyes2018}. This result was generalized to skew PBW extensions over weak $(\Sigma,\Delta)$-compatible and NI rings as the following proposition shows.

\begin{proposition}\label{prop.producto.nil1}
    Let $A=\sigma(R)\langle x_1,\dots,x_n\rangle$ be a skew PBW extension over a weak $(\Sigma,\Delta)$-compatible and NI ring $R$.
    \begin{enumerate}
        \item [\rm (1)] \cite[Theorem 4.7]{ReyesSuarez2020}. If $f=\sum_{i=0}^m a_iX_i$ and $g=\sum_{j=0}^t b_jY_j$ are elements of $A$, then $fg\in N(A)$ if and only if $a_ib_j\in N(R)$, for all $i,j$.
        \item [\rm (2)] \cite[Theorem 4.9]{ReyesSuarez2020}. For every idempotent element $e\in R$ and a fixed $i$, we have $\delta_i(e)\in N(R)$ and $\sigma_i(e)=e+u$, where $u\in N(R)$.
    \end{enumerate}
\end{proposition}

\begin{proposition}[{\cite[Theorem 3.3]{SuarezChaconReyes2022}}]\label{th.NIiffNI}
Let $A=\sigma(R)\langle x_1,\dots,x_n\rangle$ be a skew PBW extension over a weak $(\Sigma,\Delta)$-compatible ring. $R$ is NI if and only if $A$ is NI. In this case, it is clear that $N(A) = N(R)A$.
\end{proposition}

\section{Elements, Gelfand and strongly harmonic rings}\label{originalresults}

\subsection{Von Neumman regular and clean elements}\label{Elements} 

Following Lam \cite{Lam1991}, for a ring $R$, an element $a\in R$ is called {\em von Neumann regular} if there exists an element $r\in R$ with $a = ara$. The element $a\in R$ is said to be a $\pi$-{\em regular element} of $R$ if $a^{m}ra^{m} = a^{m}$, for some $r\in R$ and $m\ge 1$. We consider the set of idempotent elements ${\rm Idem}(R)$, the set of von Neumann regular elements ${\rm vnr}(R)$, and the set of $\pi$-regular elements $\pi-r(R)$. It is clear that ${\rm Idem}(R)\subseteq {\rm vnr}(R)\subseteq \pi - r(R)$. Additionally, a ring $R$ is called {\em von Neumann regular}, if the equality ${\rm vnr}(R) = R$ holds. If $\pi-r(R) = R$, then $R$ is said to be $\pi$-{\em regular}. Finally, $R$ is called {\em Boolean}, whenever ${\rm Idem}(R) = R$. In this way, the implications Boolean $\Rightarrow $ von Neumann regular $\Rightarrow$ $\pi$-regular hold.

\medskip

Contessa \cite{Contessa1984} introduced the notion of von Neumann local. An element $a\in R$ is called a {\em von Neumann local element}, if either $a\in {\rm vnr}(R)$ or $1-a\in {\rm vnr}(R)$. Following Nicholson \cite{Nicholson1977}, an element $a\in R$ is a {\em clean element}, if $a$ is the sum of a unit and an idempotent of $R$. Let ${\rm vnl}(R)$ be the set of von Neumann local elements and ${\rm cln}(R)$ the set of clean elements. If ${\rm cln}(R) = R$, then $R$ is called a {\em clean ring} \cite{Nicholson1977}. Examples of clean rings are the exchange rings and semiperfect rings. Several characterizations of clean elements have been established by different authors (see \cite{KanwarLeroyMatczuk2015} and \cite{NicholsonZhou2005}). Finally, if ${\rm vnl}(R) = R$, we say that $R$ a {\em von Neumann local} ring \cite{Hashemi}. 

\medskip

In this section, we characterize types of elements of a skew PBW extension over a weak $(\Sigma,\Delta)$-compatible and NI ring. We formulate analogue results to the
obtained for the case of skew polynomial rings \cite{Hashemi}, and the skew PBW extensions over right duo rings \cite{Hamidizadehetal2020}. 

\medskip

Proposition \ref{WeakTheorem4.5} generalizes \cite[Theorem 4.5]{Hamidizadehetal2020}.

\begin{proposition}\label{WeakTheorem4.5}
Let $A = \sigma(R)\left \langle x_1, \dots , x_n \right \rangle$ be a skew PBW extension over a weak $(\Sigma,\Delta)$-compatible and $NI$ ring $R$, and consider $f = \sum_{i=0}^la_iX_i$ and $g = \sum_{j=0}^mb_jY_j $ non-zero elements of $A$ such that $fg = c \in R$. If $b_0$ is a unit of $R$, then $a_1, a_2,\dots,a_l$ are  nilpotent elements of $R$.
\end{proposition}
\begin{proof}
Assume that $b_0$ is a unit element of $R$. Let us show that $a_1, a_2,\dots,a_l$ are all nilpotent. Since $R$ is NI ring and weak $(\Sigma, \Delta)$-compatible, we have $N(R)$ is a $(\Sigma, \Delta)$-invariant ideal of $R$. Hence $\overline{R} = R/N(R)$ is a reduced ring and also weak $(\overline{\Sigma}, \overline{\Delta})$-compatible. By Proposition \ref{prop.rigid.and.weak}, $\overline{R}$ is a $\overline{\Sigma}$-rigid ring. Since $fg = c \in R$, we have  $\overline{f}\overline{g} = \overline{c}$ in $\sigma(\overline{R}) \left \langle x_1, \dots, x_n \right \rangle$, and hence  $\overline{a_0}\overline{b_0} = \overline{c}$ and $\overline{a_i}\overline{b_j} = \overline{0}$, for each $i + j \geq 1$, by \cite[Proposition 4.2]{Hamidizadehetal2020}. Therefore, we get $\overline{a_i} = 0$ for each $i \geq 1$, since $b_0$ is a unit, whence $a_i$ is nilpotent for every $i \geq 1$.
\end{proof}


We establish the following characterization of the units of a skew PBW extension. Theorem \ref{th.units} generalizes \cite[Theorem 4.7]{Hamidizadehetal2020}.

\begin{theorem}\label{th.units} If $A=\sigma(R)\langle x_1,\dots,x_n\rangle$ is a skew PBW extension over a weak $(\Sigma,\Delta)$-compatible and NI ring $R$, then an element $f=\sum_{i=0}^ma_iX_i\in A$ is a unit of $A$ if and only if $a_0$ is a unit of $R$ and $a_i$ is nilpotent, for every $1 \le i \le m$.  
\end{theorem}
\begin{proof}
Suppose that $R$ is a weak $(\Sigma, \Delta)$-compatible and NI ring. This implies that $N(R)$ is a $(\Sigma, \Delta)$-invariant ideal of $R$, whence $\overline{R} = R/N(R)$ is reduced and weak $(\overline{\Sigma},\overline{\Delta})$-compatible. Proposition \ref{prop.rigid.and.weak} implies that $\overline{R}$ is $\overline{\Sigma}$-rigid, and $\overline{A}=A/N(A)$ is a skew PBW extension over $\overline{R}$ by Proposition \ref{prop.invariant}. 

Consider $f = \sum_{i=0}^la_iX_i$ a unit element of $A$. There exists $g=\sum_{j=0}^mb_jY_j \in A$ such that $fg = gf = 1$, which implies that $\overline{f}\overline{g}=\overline{g}\overline{f}= \overline{1}$ in $\overline{A}$, and so $\overline{a_0}\overline{b_0}=\overline{b_0}\overline{a_0}=\overline{1}$ and $\overline{a_i}\overline{b_j}=0$, for each $i+j\geq 1$ by \cite[Proposition 4.2]{Hamidizadehetal2020}. Then $\overline{a_0}$ and $\overline{b_0}$ are units of $\overline{R}$ and $a_1, \dotsc, a_l \in N(R)$. Since $N(R) \subseteq J(R)$ and $\overline{a_0}$ is a unit element of $\overline{R}$, we have that $a_0 \in U(R)$.

Conversely, let $a_0$ be a unit element and $a_1, \dotsc, a_l$ be nilpotent elements of $R$. Then $\sum_{i=1}^la_iX_i \in N(A)$ by Proposition \ref{prop.nilp}. Also, we get $N(A) \subseteq J(A)$ since $A$ is NI, and so $\sum_{i=1}^la_iX_i \in J(A)$. Therefore, we have $f =\sum_{i=0}^la_iX_i$ is a unit element of $A$.  
\end{proof}

As a consequence of the characterization of the units and nilpotent elements of a skew PBW extension over weak $(\Sigma, \Delta)$-compatible rings, we obtain Corollary \ref{WeakCorollary4.8} and Proposition \ref{th.NJ}.

\begin{corollary}\label{WeakCorollary4.8}
If $A=\sigma(R)\langle x_1,\dots,x_n\rangle$ is a skew PBW extension over a weak $(\Sigma,\Delta)$-compatible and NI ring $R$, then $U(A)=U(R)+N(R)A$.
\end{corollary}

\begin{proposition}\label{th.NJ}
If $A=\sigma(R)\langle x_1,\dots,x_n\rangle$ is a skew PBW extension over a weak $(\Sigma,\Delta)$-compatible and NI ring $R$, then $A$ is NJ.
\end{proposition}
\begin{proof}
Since $A$ is a NI ring, then $N(A)\subseteq J(A)$ by Proposition \ref{th.NIiffNI}. Additionally, if $f$ is an element of $J(A)$ with $f=\sum_{i=0}^ma_iX_i$, we obtain $1+fx_n=1+\sum_{i=0}^ma_iX_ix_n$ is a unit element of $A$. Theorem \ref{th.units} shows that the coefficients $a_0,a_1,\dots,a_m\in N(R)$, and hence $f\in N(A)$ by Proposition \ref{prop.nilp}. Thus, we conclude $N(A)=J(A)$.
\end{proof}

About idempotent elements of skew PBW extensions over weak $(\Sigma, \Delta)$-compatible NI rings, the next result generalizes \cite[Theorem 4.9]{Hamidizadehetal2020}.

\begin{theorem}\label{theoremidem}
Let $A = \sigma(R)\left \langle x_1, \dots , x_n \right \rangle$ be a skew PBW extension over a weak $(\Sigma, \Delta)$-compatible and $NI$ ring $R$. If $f = \sum_{i=0}^la_iX_i$ is an idempotent element of $A$, then $a_i \in N(R)$, for each $1 \leq i \leq l$, and there exists an idempotent element $e \in R$ such that $\overline{a_0} = \overline{e} \in R/N(R)$.
\end{theorem}
\begin{proof}
Let $f = \sum_{i=0}^la_iX_i$ be an idempotent element of $A$ and consider the element $g=1-f=(1-a_0)-\sum_{i=1}^la_iX_i$. If $f$ is an idempotent, then $fg=0\in N(R)$. Hence, by Proposition \ref{prop.producto.nil1} (1), we have $a_ia_i=a_i^2 \in N(R)$ for $1 \le i \le l$ and $a_0(1-a_0)\in N(R)$. The former means that $a_i\in N(R)$ for $i\geq 1$, and the last assertion implies that $a_0-a_0^2\in N(R)$. By \cite[Theorem 21.28]{Lam1991}, there exists an idempotent $e\in R$ such that $a_0-e\in N(R)$, that is, $\overline{a_0} = \overline{e} \in R/N(R)$.
\end{proof}

Before studying the idempotent elements, we present two preliminary results that are used in the proof of the theorem that describes of these elements in skew PBW extensions over Abelian NI rings.

\begin{lemma}\label{lemaidem}
Let $R$ be any ring, $f,e\in {\rm Idem}(R)$  and $s\in N(R)$. If $f = e+s$ and $es = se$, then $s=0$. 
\end{lemma}
\begin{proof}
If $s\neq 0$ then there exists $k\geq 2$ such that $s^k=0$ and $s^{k-1}\neq 0$. Since $f$ is idempotent, $0 = f(1-f) = (e+s)(1-e-s) = s-2es-s^2$. Thus $s^2 = (1-2e)s$, and multiplying by $s^{k-2}$ we have $0 = s^k=(1-2e)s^{k-1}$. Since $1-2e$ is invertible, it follows that $0 = s^{k-1}$, which is a contradiction. Hence $n=0$.
\end{proof}

\begin{proposition}\label{prop.idem.centr}
Let $A=\sigma(R)\langle x_1,\dots,x_n\rangle$ be a skew PBW extension over a weak $(\Sigma,\Delta)$-compatible and Abelian NI ring $R$. If $e\in {\rm Idem}(R)$, then $e\in Z(A)$.
\begin{proof}
Fix $1\leq i\leq n$. By using Proposition \ref{prop.producto.nil1} (2), there exists $u\in N(R)$ such that $\sigma_i(e)=e+u$. On the other hand, since $R$ is Abelian, we have $eu=ue$ and $\sigma_i(e)\in {\rm Idem}(R)$, which implies that $u=0$ and $\sigma_i(e)=e$ by Lemma \ref{lemaidem}. Hence, $\delta_i(e)=\delta_i(e^2)=\sigma_i(e)\delta_i(e)+\delta_i(e)e=2e\delta_i(e)$, i.e., $(1-2e)\delta_i(e)=0$, whence $\delta_i(e)=0$. Finally, since $\sigma_i(e)=e$ and $\delta_i(e)=0$, for all $1\leq i\leq n$, then $e$ commutes with the $x_i$'s, and therefore $e\in Z(A)$.
\end{proof}
\end{proposition}

Next, we formulate a result that describes the idempotent elements of skew PBW extensions over Abelian NI rings. Theorem \ref{Proposition1.3.3.a} generalizes \cite[Theorem 4.10]{Hamidizadehetal2020}.

\begin{theorem}\label{Proposition1.3.3.a}
Let $A = \sigma(R)\left \langle x_1, \dotsc, x_n \right \rangle$ be a skew PBW extension over a weak $(\Sigma,\Delta)$-compatible and Abelian $NI$ ring $R$, and $f=\sum_{i=0}^na_iX_i$ an element of $A$. If $f^2=f$, then $f= a_0\in {\rm Idem}(R)$.
\end{theorem}
\begin{proof}
Let $f=\sum_{i=0}^la_iX_i \in A$ such that $f^2=f$. By using Proposition \ref{theoremidem}, we have $a_i \in N(R)$, for each $1 \leq i \leq l$, and there exist an idempotent element $e \in R$ and a nilpotent element $b \in R$ such that $a_0 = e + b$. In this way, if we consider $h = b + \sum_{i=1}^la_iX_i \in A$, then $f = e+h$. Finally, we have $h \in N(A)$ by Proposition \ref{prop.nilp} and $he = eh$ by using Proposition \ref{prop.idem.centr}. Therefore, we conclude $h=0$ by Lemma \ref{lemaidem}, and so $f = a_0$.
\end{proof}

The next corollary extends \cite[Corollaries 4.11 and 4.12]{Hamidizadehetal2020}.

\begin{corollary}\label{WeakCorollary4.11}
If $A=\sigma(R)\langle x_1,\dots,x_n\rangle$ is a skew PBW extension over a weak $(\Sigma,\Delta)$-compatible and Abelian NI ring $R$, then ${\rm Idem}(A)={\rm Idem}(R)$, and so $A$ is an Abelian ring.
\end{corollary}

In \cite[Theorem 4.14]{Hamidizadehetal2020} the von Neumann regular elements in skew PBW extensions over right duo rings were characterized. Next, we formulate a generalization of this theorem.

\begin{theorem}\label{WeakTheorem4.14}
If $A=\sigma(R)\langle x_1,\dots,x_n\rangle$ is a skew PBW extension over a weak $(\Sigma,\Delta)$-compatible and Abelian NI ring $R$, then ${\rm vnr}(A)$ consists of the elements of the form $\sum_{i=0}^ma_iX_i$ where $a_0=ue$, $a_i\in e(N(R))$, for every $i\geq 1$, for some $u\in U(R)$ and $e\in {\rm Idem}(R)$.
\end{theorem}
\begin{proof}
By Corollary \ref{WeakCorollary4.11}, $A$ is an Abelian ring. Additionally, $f\in{\rm vnr(A)}$ if and only if $f = ue$, for some $u\in U(A)$ and $e\in {\rm Idem}(A)$ \cite[Proposition 4.2]{Hashemi}. Hence, the result follows from Corollaries \ref{WeakCorollary4.8} and \ref{WeakCorollary4.11}.
\end{proof}

As a consequence of \cite[Theorem 1, Lemma 2]{Badawi}, we obtain the following description of the $\pi$-regular elements over Abelian rings.

\begin{proposition}\label{badawi}
    Let $R$ be an Abelian ring. Then $r$ is a $\pi$-regular element of $R$ if and only if there exist $e\in {\rm Idem}(R)$ and $u\in U(R)$ such that $er = eu$ and $(1-e)r\in N(R)$.
\end{proposition}

As a result of the previous proposition, we can describe the $\pi$-regular elements in a skew PBW extension over weak $(\Sigma,\Delta)$-compatible and Abelian NI ring. Theorem \ref{WeakTheorem4.15} generalizes \cite[Theorem 4.15]{Hamidizadehetal2020}. 

\begin{theorem}\label{WeakTheorem4.15}
If $A=\sigma(R)\langle x_1,\dots,x_n\rangle$ is a skew PBW extension over a weak $(\Sigma, \Delta)$-compatible and Abelian NI ring $R$, then
\begin{center}
    $\displaystyle \pi - r(A) = \left \{ \sum a_iX_i\in A\ |\ a_0\in \pi-r(R),\ a_i\in N(R), \ \text{for}\ i \geq 1 \right \}$.
\end{center}
\end{theorem}
\begin{proof}
    Let $f=\sum_{i=0}^la_iX_i$ be an element of $\pi-r(A)$. Since $A$ is Abelian, there exist elements $e\in{\rm Idem}(A)={\rm Idem}(R)$ and $u\in U(A)$ such that $ef=eu$ and $(1-e)f\in N(A)$, by Proposition \ref{badawi}. Then $ea_0=eu'$ and $ea_i\in N(R)$, for some $u'\in U(R)$ and for all $1\le i\le l$ by Theorem \ref{th.units}. Additionally, we have $(1-e)a_i\in N(R)$ for all $0\le i \le l$ by Proposition \ref{prop.nilp}. Thus, Proposition \ref{badawi} shows that $a_0\in\pi-r(R)$   and $a_i\in N(R)$, for $1\le i\le l$.

    On the other hand, suppose that $a_0\in \pi-r(R)$ and $a_i\in N(R)$, for all $i\ge 1$. By using Proposition \ref{badawi}, there exist $e\in{\rm Idem}(R)$ and $u\in U(R)$ such that $ea_0=eu$ and $(1-e)a_0\in N(R)$. This implies that $ef = ea_0+\sum_{i=1}^l  ea_iX_i=e\left(u+\sum_{i= 1}^n a_iX_i\right) = eu'$ and $(1-e)f = \sum_{i=0}^l (1-e)a_iX_i$ where $u' = u+\sum_{i= 1}^l a_iX_i$. Since $N(R)$ is an ideal of $R$, $(1-e)a_i\in N(R)$ and therefore $(1-e)f\in N(A)$ by Proposition \ref{prop.nilp}, and $u'\in U(A)$ by Theorem \ref{th.units}. Therefore, Proposition \ref{badawi} guarantees that $f\in\pi-r(A)$.
\end{proof}


Theorem \ref{WeakTheorem4.16} extends \cite[Theorem 4.16]{Hamidizadehetal2020}.

\begin{theorem}\label{WeakTheorem4.16}
If $A=\sigma(R)\langle x_1,\dots,x_n\rangle$ is a skew PBW extension over a weak $(\Sigma, \Delta)$-compatible and Abelian NI ring $R$, then ${\rm vnl}(A)$ consists of the elements of the form $\sum_{i=0}^la_iX_i$, where either $a_0 = ue$ or $a_0 = 1 - ue$, $a_i \in eN(R)$, for every $i \geq 1$, some element
$u \in U(R)$ and $e \in {\rm Idem}(R)$.
\end{theorem}
\begin{proof}
It follows from Theorem \ref{WeakTheorem4.14} and \cite[Theorem 6.1 (2)]{Hashemi}.
\end{proof}

In \cite[Theorem 4.17]{Hamidizadehetal2020}, the clean elements of skew PBW extensions over right duo rings were characterized. We present a generalization of this result.

\begin{theorem}\label{WeakTheorem4.17}
Let $A=\sigma(R)\langle x_1,\dots,x_n\rangle$ be a skew PBW extension over a weak $(\Sigma,\Delta)$-compatible and Abelian NI ring $R$. Then
\[{\rm cln}(A)=\left\lbrace\sum a_iX_i\in A\ |\ a_0\in{\rm cln}(R),\ a_i\in N(R)  \right\rbrace. \]
\end{theorem}
\begin{proof}
The result follows from  Proposition \ref{th.NIiffNI} and Corollaries \ref{WeakCorollary4.8} and \ref{WeakCorollary4.11}.
\end{proof}

\subsection{Strongly harmonic and Gelfand rings}\label{Gelfandrings}Mulvey \cite{Mulvey1979a} introduced {\it strongly harmonic rings} with the aim of generalizing the Gelfand duality from $C^{*}$-algebras to  rings (not necessarily commutative). Borceux and Van den Bossche \cite{Borceux1983} modified the definition of strongly harmonic rings and defined the {\it Gelfand rings}. A ring $R$ is called \textsl{Gelfand} (resp. {\it strongly harmonic}) if for each pair of distinct maximal right ideals (resp. maximal ideals) $M_1$, $M_2$ of $R$, there are right ideals (resp. ideals) $I_1$, $I_2$ of $R$ such that $I_1\not\subseteq M_1$, $I_2\not\subseteq M_2$ and $I_1I_2=0$. Equivalently, $R$ is a Gelfand ring (resp. strongly harmonic) if for each pair of distinct maximal right ideals (resp. maximal ideals) $M_1$, $M_2$ of $R$, there are elements $r\notin M_1$, $s\notin M_2$ of $R$ such that $rRs=0$. Gelfand rings and strongly harmonic rings have been investigated by different authors such as Borceux and Van den Bossche \cite{Borceux1983}, Borceux et al., \cite{Borceuxetal1984}, Carral \cite{Carral1980}, Demarco and Orsatti \cite{DemarcoOrsatti1971}, Mulvey \cite{Mulvey1976, Mulvey1979a, Mulvey1979b}, Sun \cite{Sun1992a, Sun1992b}, Zhang et al., \cite{Zhangetal2006}. We continue the study of Gelfand and Harmonic rings in the setting of skew PBW extensions over weak $(\Sigma, \Delta)$-compatible rings. 

\begin{proposition}\label{prop.strongly.coef}
Let $A=\sigma(R)\langle x_1,\dots,x_n\rangle$ be a skew PBW extension over a weak $(\Sigma,\Delta)$-compatible and NI ring $R$. Then $A/J(A)$ is a Gelfand ring {\rm (}resp. strongly harmonic{\rm )} if and only if for each pair of distinct maximal right ideals (resp. maximal ideals) $M_1$, $M_2$ of $A$, there exist elements of $R$, $r\notin M_1$ and $s\notin M_2$ such that $rRs\subseteq N(R)$.
\end{proposition}
\begin{proof} Let $M_1$, $M_2$ be a pair of distinct maximal right ideals of $A$. Since $A/J(A)$ is a Gelfand ring, there exist elements $f, g$ of $A$ such that $f=\sum_{i=0}^ma_iX_i\notin M_1$, $g=\sum_{j=0}^lb_jX_j\notin M_2$ and $fAg\subseteq J(A)=N(A)$. Since $f\notin M_1$, $a_t\notin M_1$ for some coefficient of $f$ with $1 \le t \le m$. In this way, we also have $b_s\notin M_2$ for some coefficient of $g$ with $1 \le s \le l$. Since $fRg \subseteq fAg \subseteq N(A)$, then $fcg \in N(A)$ and $a_icb_j\in N(R)$ for all $i,j$, and for every $c \in R$ by Proposition \ref{prop.nilp}. Thus, if we consider $r=a_t$ and $s=b_s$, the result follows. 

For the other implication, let $M_1$, $M_2$ be a pair of distinct maximal right ideals of $A$. Since $rRs\subseteq N(R)$, for some $r \notin M_1$ and $s \notin M_2$, then $rAs\subseteq N(A)$, by Proposition \ref{prop.nilp}. This implies that $A/J(A)$ is a Gelfand ring. 

The proof of the strongly harmonic case is analogous.
\end{proof}

\begin{proposition}\label{prop.not.local}
If $A=\sigma(R)\langle x_1,\dots,x_n\rangle$ is a skew PBW extension over a weak $(\Sigma,\Delta)$-compatible and NI ring, then $A$ is not a local ring.
\end{proposition}
\begin{proof}
Suppose that $A$ is a local ring. By using Theorem \ref{th.units}, we have $x_n$ is not an unit of $A$. In this way, since $J(A)=N(A)$ we have that $x_n \in N(A)$, which contradicts Proposition \ref{prop.nilp}. Hence $A$ is not local.
\end{proof}

\begin{example} The following examples illustrate the above theorem.
    \begin{itemize}
       \item[\rm (i)] In the Example \ref{exampleweak1}, we have that the Ore extension $R_2[x;\sigma,\delta]$ is not a local ring by Proposition \ref{prop.not.local}.
       \item[\rm (ii)] Consider the skew PBW extension $A = \sigma({\rm S}_2(\mathbb{Z}))\langle x,y,z \rangle$ over the ring ${\rm S}_2(\mathbb{Z})$ in Example \ref{exampleweak2}. Proposition \ref{prop.not.local} shows that $A$ is not a local ring. 
    \end{itemize}
\end{example} 

\begin{theorem}\label{th.no.gelfand}
Let $A=\sigma(R)\langle x_1,\dots,x_n\rangle$ be a skew PBW extension over a weak $(\Sigma,\Delta)$-compatible and NI ring $R$. If $N(R)$ is a prime ideal of $R$ then $A/J(A)$ is not a Gelfand ring.
\end{theorem}
\begin{proof}
Suppose that $A/J(A)$ is a Gelfand ring. By using Proposition \ref{prop.not.local}, there exist at least two maximal right ideals $M_1,M_2$ of $A$. Additionally, since $A/J(A)$ is a Gelfand ring, there exist $r,s\in R$ such that $r\notin M_1$, $s\notin M_2$ and $rRs\subseteq N(R)$ by Proposition \ref{prop.strongly.coef}. Since $N(R)$ is a prime ideal of $R$, we have $r\in N(R)$ or $s\in N(R)$. Finally, since $N(R)\subseteq N(A)=J(A)\subseteq M_1, M_2$, then $r\in M_1$ or $s\in M_2$ which is a contradiction. Hence, we conclude $A/J(A)$ is not a Gelfand ring.
\end{proof}

\begin{corollary}\label{corollary.no.gelfand}
Let $A=\sigma(R)\langle x_1,\dots,x_n\rangle$ be a skew PBW extension over a weak $(\Sigma,\Delta)$-compatible and NI ring $R$. If $N(R)$ is a prime ideal then $A/N_*(A)$ is not a Gelfand ring.
\end{corollary}

\begin{example} \label{examplenotGelfand}
If we consider the Theorem \ref{th.no.gelfand} and the Corollary \ref{corollary.no.gelfand}, we conclude $A/J(A)$ and $A/N_{*}(A)$ are not Gelfand rings where $A$ is the Ore extension $R_2[x;\sigma,\delta]$ over the ring of upper triangular matrices $R_2$ in Example \ref{exampleweak1}, or the skew PBW extension $A = \sigma({\rm S}_2(\mathbb{Z}))\langle x,y,z \rangle$ over the ring ${\rm S}_2(\mathbb{Z})$ in Example \ref{exampleweak2}.
\end{example}

\begin{theorem}\label{theorem.harmonic.unique}
    Let $A=\sigma(R)\langle x_1,\dots,x_n\rangle$ be a skew PBW extension over a weak $(\Sigma,\Delta)$-compatible and NI ring. If $N(R)$ is a prime ideal then $A/J(A)$ is strongly harmonic if and only if $A$ has a unique maximal ideal.
\end{theorem}
\begin{proof}
    If $A$ has a unique maximal ideal, then $A/J(A)$ has a unique maximal ideal and therefore $A/J(A)$ is strongly harmonic. If $A$ has at least two maximal ideals, the proof of the Theorem \ref{th.no.gelfand} guarantees that $A/J(A)$ is not a strongly harmonic ring.
\end{proof}


\section{Future work}\label{future}

Different authors have studied the notion of compatibility for the study of modules over Ore extensions and skew PBW extensions (e.g., \cite{AlhevazMoussavi2012, Annin2004, NinoRamirezReyes2020, Reyes2019}). Considering the results obtained in this paper, we think as future work to investigate a classification of several types of elements in modules extended over skew PBW extensions. Also, following this idea and the notions of {\em strongly harmonic and Gelfand modules} introduced by Medina-B\'arcenas et al. \cite{Medinaetal2020}, it will be interesting to study these modules for these noncommutative rings.

\medskip

On the other hand, since skew PBW extensions are examples of noncommutative algebraic structures such as the {\em semi-graded rings} defined by Lezama and Latorre \cite{Lezamalatorre2017}, which have been studied recently from the point of view of noncommutative algebraic geometry (topics such as Hilbert polynomial and Hilbert series, generalized Gelfand-Kirillov dimension, noncommutative schemes, and Serre-Artin-Zhang-Verevkin theorem, and others, see  \cite{Lezama2020, Lezamagomez2019}), and having in mind that semi-graded rings are more general than $\mathbb{N}$-graded rings, the natural task is to investigate the several types of elements studied in this paper and the notions of strongly harmonic and Gelfand rings in this more general setting.

\end{document}